\documentclass[journal]{IEEEtran}
\usepackage{cite}
\usepackage{amssymb}

\newtheorem{secthm}{Theorem}[section]

\newtheorem{secexa}[secthm]{Example}
\newtheorem{secprop}[secthm]{Proposition}
\newtheorem{secdefn}[secthm]{Definition}
\newtheorem{secrem}[secthm]{Remark}

\newcommand{\cK} { {\cal K}}
\newcommand{\cH} { {\cal H}}
\newcommand{\bR} { {\bf R}}
\newcommand{\bC} { {\bf C}}

\usepackage[dvipdfm]{graphicx,color}
\def\red{\hfill $\lhd$}
\begin{document}
\title{Nonlinear Eigenvalue Approach to Differential Riccati Equations for Contraction Analysis}
\author{Yu~Kawano and
        Toshiyuki~Ohtsuka% <-this % stops a space        
\thanks{Y. Kawano and T. Ohtsuka are with the Department of Systems Science, Graduate School of Informatics,
        Kyoto University, Sakyo-ku, Kyoto 606-8501, Japan
        (e-mail: ykawano@i.kyoto-u.ac.jp; ohtsuka@i.kyoto-u.ac.jp).}% <-this % stops a space
\thanks{This work was partly supported by JSPS KAKENHI Grant Numbers
	15K18087 and 15H02257 and JST CREST.}
}

\maketitle

\begin{abstract}
In this paper, we extend the eigenvalue method of the algebraic Riccati equation to the differential Riccati equation (DRE) in contraction analysis. One of the main results is showing that solutions to the DRE can be expressed as functions of nonlinear eigenvectors of the differential Hamiltonian matrix. Moreover, under an assumption for the differential Hamiltonian matrix, real symmetricity, regularity, and positive semidefiniteness of solutions are characterized by nonlinear eigenvalues and eigenvectors.
\end{abstract}
\begin{keywords}
Nonlinear systems, differential Riccati equations, nonlinear eigenvalues, contraction analysis.
\end{keywords}

\IEEEpeerreviewmaketitle

%%%%%%%%%%%%%%%%%%%%%%%%%%%%%%%%%%%%%%%%%%%%%%%%%%%%%%%%%%%%%%%%%%%%%%%%%%%%%%%%
\section{Introduction}
In this paper, we present a novel eigenvalue method for the differential Riccati equation (DRE) in contraction analysis. Contraction and incremental analysis have been studied intensively in recent decades, which deals with trajectories of nonlinear systems with respect to one another \cite{FS:14,MS:13,MS:14,Schaft:13,FS:13,KS}. One of the interesting ideas of contraction theory is considering the infinitesimal metric instead of a feasible distance function by lifting of functions and vector fields on manifold to their tangent and cotangent bundles, which is one of the differences from classical nonlinear geometric control theory. In such theoretical frameworks, for instance, the Lyapunov theorem \cite{FS:14}, optimal control \cite{MS:13}, $H^{\infty}$ control \cite{MS:14}, dissipativity \cite{Schaft:13,FS:13}, and balanced truncation \cite{KS} have been studied. 

In the optimal control in the contraction framework, a kind of Riccati equation that we call a DRE plays an important role, which is a nonlinear partial differential equation for an unknown matrix whose elements are functions of the state and time. The DRE can be viewed as an extension of algebraic and differential Riccati equations for linear time-invariant and variant systems rather than as a Hamilton-Jacobi equation (HJE). 

One of the most important analysis methods for the algebraic Riccati equation (ARE) is the eigenvalue method \cite{pot,mac,lau}. This method shows that solutions to the ARE, a nonlinear algebraic equation, can be described as functions of eigenvectors of the Hamiltonian matrix, and in terms of eigenvalues and eigenvectors, real symmetricity, regularity, and positive semidefiniteness of solutions have been studied. This method has been extended to the DRE for linear periodic systems \cite{pastor1993differential,kano1979periodic,razzaghi1979computational}, which is different from the equation considered in this paper. 

Our main concern in this paper is extending the eigenvalue method to the DRE in contraction analysis in terms of recently introduced nonlinear eigenvalues and eigenvectors \cite{KO:15,label6965}. First, we demonstrate that solutions to the DRE can be expressed as functions of nonlinear right eigenvectors of the corresponding Hamiltonian matrix as in the linear case. Next, we investigate its solution structures when nonlinear right eigenvectors of the Hamiltonian matrix span the entire space. In this case, a nonlinear right eigenvalue is also a left eigenvalue and vice versa, and if $\lambda$ is an (right or left) eigenvalue of the Hamiltonian matrix, then $-\lambda$, the complex conjugate of $\lambda$ denoted by $\lambda^*$, and $-\lambda^*$ are also eigenvalues similarly to the linear case. Moreover, we study real symmetricity, regularity, and positive semidefiniteness of solutions to the DRE in terms of nonlinear eigenvalues.

The nonlinear left and right eigenvalues and eigenvectors of the Jacobian matrix of a vector field correspond to a one-dimensional invariant distributions and codistributions, respectively, in the time-invariant case. A similar concept can be found in Koopman operator theory \cite{mauroy2014global}. The Koopman eigenfunction coincides with an invariant space under the Lie derivative of a function. The concepts of nonlinear eigenvalues are originally introduced in non-commutative algebra in relation to the pseudo-linear transformation (PLT) \cite{label3043,label9353}. The PLT can be interpreted as a generalized notion of linear transformation to differential one-forms. Non-commutative algebra and the PLT are used for analysis of linear time-varying and nonlinear control systems \cite{ilchmann2005algebraic,label1618}. In contrast to nonlinear systems, there is no application of such eigenvalues to linear time-varying systems. Some papers on the DRE of periodic systems \cite{pastor1993differential,kano1979periodic,razzaghi1979computational} do not use such eigenvalues of the Hamiltonian matrix. Those papers use eigenvalues of the transition matrix of the Hamiltonian matrix in the sense of linear algebra instead.

{\bf Notations:} Let $\bR$ and $\bC$ be the fields of real and complex numbers, respectively. Let $\cK_{\bR}$ be the field of the real meromorphic functions in variables $x_1,x_2,\dots,x_n,t$. Let $\cK$ be the set of functions $\{a+bj:a,b\in\cK_{\bR}\}$, where $j$ is the imaginary unit, and the domain of definition of both $a$ and $b$ is $\bR^n\times \bR$. Note that $\cK_{\bR}\subset \cK$, and $\cK$ is a field. Then, $\cK^{2n}$ is a vector space over $\cK$. The reason we consider (the not commonly used) field $\cK$ is that we exploit a concept of nonlinear eigenvalue of matrix $A\in \cK^{n\times n}$. As will be shown in Example~\ref{lieig}, for a constant matrix $M\in\bC^{n\times n}$, the set of nonlinear eigenvalues contains the set of eigenvalues in linear algebra. Since a linear eigenvalue can be a complex number even if $M$ is in $\bR^{n\times n}$, a nonlinear eigenvalue can be an element in $\cK$ even if matrix $A$ is an element in $\cK_{\bR}^{n\times n}$. Therefore, we consider field $\cK$ in this paper.

For a scalar-valued function $V(x,t)\in\cK$, we denote a row vector consisting of the partial derivatives of $V$ with respect to $x_i$ $(i=1,2,\dots,n)$ as $\partial V/\partial x$, and we denote $\partial^{\rm T} V/\partial x:=(\partial V/\partial x)^{\rm T}$.  For matrix $A(x,t)\in\cK^{r_1\times r_2}$, ${\rm rank}_{\cK}A(x,t)=r$ means that the rank of $A(x,t)$ over field $\cK$ is $r$. In particular, if $r_1=r_2=r$, $A$ is said to be regular.

Next, we introduce an operator $\delta_f:\cK\to\cK$. By using real analytic vector-valued function $f(x,t):\bR^n\times\bR\to\bR^n$, operator $\delta_f:\cK\to\cK$ is defined as
\begin{eqnarray}
\delta_f(a(x,t))=\frac{\partial a(x,t)}{\partial t} +\frac{\partial a(x,t)}{\partial x}f(x,t), \ a(x,t)\in\cK.\label{delta}
\end{eqnarray}
Field $\cK$ is a differential field with respect to $\delta_f$. For matrix $X(x,t)=(X_{ij}(x,t))\in \cK^{n\times n}$, $\delta_f (X(x,t))$ denotes the matrix whose $(i,j)$th element is $\delta_f(X_{ij}(x,t))$. Operator $\delta_f$ coincides with the time derivative of a function along a solution to $\dot x(t)=f(x(t),t)$ because the time derivative of $a(x(t),t)$ is $da(x(t),t)/dt=\partial a(x(t),t)/\partial t+(\partial a(x(t),t)/\partial x)f(x(t),t)$. In systems and control, in general, we study a real-valued vector field. Thus, we assume that $f$ is real-valued. Throughout this paper, we leave out arguments of functions when these are clear from the context.
%%%%%%%%%%%%%%%%%%%%%%%%%%%%%%%%%%%%%%%%%%%%%%%%%%%%%
\section{Eigenvalue Approach}\label{Ea:s}
\subsection{Differential Riccati Equation}
Let each element of $n\times n$ matrices $A(x,t)$, $R(x,t)=R^{\rm T}(x,t)$, and $Q(x,t)=Q^{\rm T}(x,t)$ be real analytic. In this paper, we study the following equation for unknown matrix $X(x,t)\in \cK^{n\times n}$:
\begin{eqnarray}
&&\delta_f (X(x,t))+X(x,t)A(x,t)+A^{\rm T} (x,t)X(x,t)\nonumber\\
&&-X(x,t)R(x,t)X(x,t)=-Q(x,t).\label{ric}
\end{eqnarray}
Equation (\ref{ric}) is a generalization of the algebraic Riccati equation (ARE), and thus we call (\ref{ric}) a (generalized) differential Riccati equation (DRE). A real symmetric and positive definite solution plays an important role in systems and control theory such as that in contraction analysis \cite{FS:14,MS:13}.
%------------------------------------------------------------------------
\begin{secexa}\label{contra:ex}
A stabilizing controller is designed by using a solution to a DRE.
   Consider a time-invariant real analytic system 
    \begin{eqnarray*}   
   \dot x(t) = f(x(t)) + B u(t),
   \end{eqnarray*}
   where $x\in \bR^n$ and $u\in \bR^m$. For $A= \partial f/\partial x$, $R=BB^{\rm T}$, and symmetric and positive definite $Q(x)$ at each $x\in\bR^n$, suppose that a DRE
    \begin{eqnarray}
    \delta_f(X)+ X \frac{\partial f}{\partial x}+ \frac{\partial^{\rm T} f}{\partial x} X -X B B^{\rm T}X =-Q \label{conlya}
    \end{eqnarray}
   has a symmetric and positive definite solution $X(x)$ at each $x\in\bR^n$. Here, we show that if $X$ satisfies $(\partial X_{ij}/\partial x)B=0$, and if there exists a vector-valued function $k(x)\in\cK^m$ such that $\partial k/\partial x=B^{\rm T} X$ then $u=-k(x)$ is a stabilizing controller. Under these assumptions, (\ref{conlya}) can be rearranged as
\begin{eqnarray*}
&&\delta_{f-Bk}(X) + X \frac{\partial (f - B k)}{\partial x} + \frac{\partial^{\rm T} (f-B k)}{\partial x} X\\
&&=-Q -X B B^{\rm T}X.
\end{eqnarray*}
We notice that  $V(x,\delta x):=\delta x^{\rm T}X\delta x$ is a contraction Riemannian metric for the closed loop system and its variational system.
\begin{eqnarray}
&&\dot x=f(x)-B k(x),\label{clEx}\\
&&\frac{d}{dt}\delta x(t)=\frac{\partial (f(x)-B k(x))}{\partial x}\delta x(t).\nonumber
\end{eqnarray}
According to \cite{FS:14}, the closed loop system is incrementally globally asymptotically stable. Roughly speaking, any pair of trajectories of the closed-loop system converges to each other. If the system has an unique equilibrium point, the system is globally asymptotically stable. In summary, by solving DRE (\ref{conlya}), we can construct a stabilizing controller $u=-k=-\int B^{\rm T} X dx$. A similar result has been obtained for time-varying systems, and the integrability condition of $B^{\rm T}X$ was dropped by using a line integral \cite{MS:13}. \red
\end{secexa}
%------------------------------------------------------------------------

Other applications of the DRE are, for instance, incremental optimal control \cite{MS:13}, $L^2$-gain analysis \cite{MS:14}, and balanced truncation \cite{KS} problems. In linear systems and control theory, optimal and $H^{\infty}$ controllers are designed by solving AREs. These results are extended in the contraction framework by using DREs \cite{MS:13,MS:14}. Moreover, the so-called differential balanced realization \cite{KS} is defined by using Lyapunov types of equations, which are specific DREs for $R\equiv 0$. The differential balanced realization is used for the model reduction. In these optimal control and balanced truncation problems, symmetric and positive definite solutions to DREs are used. Since the DRE is a nonlinear partial differential equation for an unknown matrix, the structures of solutions have not been adequately studied. That is, it is unclear when a symmetric and positive definite solution exists. Here, our concern is investigating the solution structures by using nonlinear eigenvalues and eigenvectors \cite{KO:15,label6965}. That is, we extend the eigenvalue method of the ARE \cite{pot,mac,lau}.
% % % % % % % % % % % % % % % % % % % % % % % % % % % % % % % % % % % % % % % % % % % % % % %
\subsection{Generalized Hamiltonian Matrix}
Solutions to the ARE are characterized by the eigenvalues and eigenvectors of the Hamiltonian matrix. The counterpart of the Hamiltonian matrix to the DRE is
\begin{eqnarray}
\cH(x,t):=
\left[\begin{array}{cc}
A(x,t) & -R(x,t) \\
-Q(x,t) & -A^{\rm T}(x,t)
\end{array}\right].\label{HM}
\end{eqnarray}
We call this $\cH (x,t) \in \cK_{\bR}^{2n\times 2n}$ a (generalized) differential Hamiltonian matrix. Since the elements of $A$, $R=R^{\rm T}$, and $Q=Q^{\rm T}$ are real analytic, the elements of $\cH$ are also real analytic.

Next, we show the definition of the nonlinear eigenvalues and eigenvectors \cite{KO:15,label9353,label6965}.
\begin{secdefn}\label{eig:d}
Consider $\delta_f$ defined in (\ref{delta}).  Let $M\in\cK^{n\times n}$.
\begin{enumerate}
\item $v\in\cK^{n}\setminus\{0\}$ is a left eigenvector for $M$ associated with left eigenvalue $\alpha\in\cK$ if $v^{\rm T} M+\delta_f(v)^{\rm T} =v^{\rm T} \alpha$.
\item $w\in\cK^{n}\setminus\{0\}$ is a right eigenvector for $M$ associated with right eigenvalue $\beta\in\cK$ if $M w-\delta_f(w)=\beta w$. 
\end{enumerate}
Moreover, the sets of left and right eigenvalues of $M$ are denoted by ${\rm lspec}_f(M)$ and ${\rm rspec}_f(M)$, respectively.
\end{secdefn}

Nonlinear eigenvalues relate to invariant spaces when $M=\partial f/\partial x$. The definitions of left and right eigenvalues are respectively rearranged as ${\cal L}_f (v^{\rm T}dx) = \alpha (v^{\rm T}dx)$ with the Lie derivative of one-forms along $f$ and $[w,f]=\beta w$ with the Lie bracket of vector fields. Thus, $v^{\rm T} dx$ and $w$ are respectively one-dimensional invariant codistribution and distribution. 

\begin{secexa}\label{lieig}
	In the linear case when $M \in \bR^{n\times n}$, the first (or second) equation in Definition \ref{eig:d} holds for linear eigenvalue $\alpha\in\bC$ and left eigenvector $v\in\bC^n$ (or $\beta\in\bC$ and right eigenvector $v\in\bC^n$). Thus, the linear eigenvalue and eigenvector are a nonlinear eigenvalue and eigenvector.\red
\end{secexa}

Nonlinear eigenvalues have similar properties to those in linear algebra. These are invariant under the $\delta_f$-conjugacy defined below, which relates to a change of basis over a differential field. Let $\{v_1,\dots,v_n\}$ and $\{w_1,\dots,w_n\}$ be bases for $\cK^n$. Then, there exist matrices $M,N\in\cK^{n\times n}$ such that $[\delta_f(v_1) \ \dots \ \delta_f(v_n)]=M[v_1 \ \dots \ v_n]$ and $[\delta_f(w_1) \ \dots \ \delta_f(w_n)]=N [w_1 \ \dots \ w_n]$. For two bases, there exists a regular matrix $T\in\cK^{n\times n}$ such that $[v_1,\dots,v_n]=T[w_1,\dots,w_n]$. By applying $\delta_f$ from the left, we have
\begin{eqnarray*}
M [v_1,\dots,v_n]&=&\delta_f(T) [w_1,\dots,w_n] + TN [w_1,\dots,w_n]\\
&=&(TN+\delta_f(T)) T^{-1} [v_1,\dots,v_n].
\end{eqnarray*}
Since $\{v_1,\dots,v_n\}$ is a basis, we have $M=(TN+\delta_f(T))T^{-1}$. This pair of matrices $(M,N)$ is said to be $\delta_f$-conjugate.
%---------------------------------------
\begin{secdefn}\label{dconj}\cite{label3043,label9353}
A pair of matrices $(M,N)\in \cK^{n\times n} \times \cK^{n\times n}$ is $\delta_f$-conjugate (with respect to $T$) if there exists a regular matrix $T \in \cK^{n\times n}$ such that $M=(TN+\delta_f(T))T^{-1}$ holds.
\end{secdefn}
%---------------------------------------
	\begin{secexa}
		When $n=1$, we have the definition of $\delta_f$-conjugation for elements in $a,b\in\cK$ \cite{label3043,label9353}.  A pair $(a,b)$ is $\delta_f$-conjugate if $b=a+\delta_f(c)/c$ for non-zero $c\in\cK$. \red
	\end{secexa}
%---------------------------------------
\begin{secprop}\label{lreig:p}\cite{label3043,label9353} Let $M$ be in $\cK^{n\times n}$.
    \begin{enumerate}
        \item Let $(a,b)\in \cK\times\cK$ be $\delta_f$-conjugate. If $a\in{\rm lspec}_f(M)$ (or $a\in{\rm rspec}_f(M)$) then $b\in{\rm lspec}_f(M)$ (or $b\in{\rm rspec}_f(M)$) .
        \item If $(M,N)$ is $\delta_f$-conjugate, ${\rm rspec}_f(M)={\rm rspec}_f(N)$ and ${\rm lspec}_f(M)={\rm lspec}_f(N)$. \red
    \end{enumerate}
\end{secprop}

	\begin{secexa}
		If $(M,N)$ is $\delta_f$-conjugate, then we have $N=(T^{-1}M+\delta_f (T^{-1}))T$, i.e. $(N,M)$ is also $\delta_f$-conjugate. \red
	\end{secexa}	
	\begin{secexa}
		If both $(L,M)$ and $(M,N)$ are $\delta_f$-conjugate with respect to regular $T,S\in \cK^{n\times n}$. Then, $(L,N)$ is also $\delta_f$-conjugate with respect to $TS$. \red
	\end{secexa}
\begin{secexa}
Consider system $\dot x =f(x)$ and its variational system $(d\delta x/dt) = (\partial f/\partial x) \delta x$. After an analytic diffeomorphic coordinate transformation $z=\varphi(x)$, we have $(d\delta z/dt) = (T(\partial f/\partial x)+\delta_f(T))T^{-1} \delta z$, where $T:=\partial \varphi/\partial x$. Proposition~\ref{lreig:p}~2) implies that $\partial f/\partial x$ and $(T(\partial f/\partial x)+\delta_f(T))T^{-1}$ have the same nonlinear left and right eigenvalues. \red
\end{secexa}

Proposition \ref{lreig:p} 1) comes from a scalar multiplication of eigenvectors. For some nonzero $a\in\cK$, left eigenvalue $\alpha$ and eigenvector $v$, we have
\begin{eqnarray*}
av^{\rm T}M+\delta_f(av^{\rm T})
&=&av^{\rm T}M+a\delta_f(v^{\rm T})+\delta_f(a)v^{\rm T}\\
&=&(\alpha+\delta_f(a)/a) av^{\rm T}.
\end{eqnarray*}
Then, $\alpha+\delta_f(a)/a$ and $av$ are also a left eigenvalue and eigenvector, respectively. Note that $(\alpha,\alpha+\delta_f(a)/a)$ is $\delta_f$-conjugate. For a similar relationship for right eigenvectors, see (\ref{dcd}) below.

% % % % % % % % % % % % % % % % % % % % % % % % % % %
\subsection{Main Theorem}
Here, we show that solutions to the DRE can be expressed as functions of nonlinear eigenvectors of the corresponding differential Hamiltonian matrix $\cH$.
\begin{secdefn}\label{HinvD}
A linear subspace $W \subset \cK^{2n}$ is said to be right $\cH$ invariant if $\cH W-\delta_f(W) \subset W$ holds. We denote the set of right eigenvalues of $\cH$ in $W$ by ${\rm rspec}_f (\cH|_W)$, i.e.,
\begin{eqnarray*}
&&{\rm rspec}_f (\cH|_W)\\
&&:=\{\beta\in\cK:\cH w-\delta_f(w)=\beta w, \ w\in W\setminus\{0\}\}.
\end{eqnarray*}
\end{secdefn}

The following is the main theorem of this paper.
% % % % % % % % % % % % % % % % % % % % % % % % % % % % % %
\begin{secthm}\label{ham:t}
Assume there exists an $n$-dimensional $\cH$ invariant subspace $W\subset \cK^{2n}$.
Consider matrices $U,V\in\cK^{n\times n}$ such that
\begin{eqnarray}
W={\rm Im}\left[\begin{array}{c}U\\ V\end{array}\right].\label{Wdef}
\end{eqnarray}
If $U$ is regular,  $X:=VU^{-1}\in\cK^{n\times n}$ is a solution to (\ref{ric}) and satisfies 
\begin{eqnarray}
{\rm rspec}_f\left(A-RX\right)={\rm rspec}_f(\cH|_W)\label{cHinveig}.
\end{eqnarray}

Conversely, if $X\in\cK^{n\times n}$ is a solution to (\ref{ric}), there exist $U,V\in\cK^{n\times n}$ such that $U$ is regular, and $X=VU^{-1}$. Moreover, for these $U$ and $V$, subspace $W\subset\cK^{2n}$ in (\ref{Wdef}) is an $n$-dimensional $\cH$ invariant subspace and satisfies (\ref{cHinveig}).
\end{secthm}
%---------------------------------------------
\begin{proof}
We prove the first parts. Since $W$ is $\cH$ invariant, there exists some matrix $\Lambda\in\cK^{n\times n}$ such that
\begin{eqnarray}
\left[\begin{array}{cc}
A & -R \\
-Q & -A^{\rm T}
\end{array}\right]
\left[\begin{array}{c}
U\\ V
\end{array}\right]
-\left[\begin{array}{c}
\delta_f(U)\\ \delta_f(V)
\end{array}\right]
=\left[\begin{array}{c}
U\\ V
\end{array}\right]
\Lambda.\label{riegh}
\end{eqnarray}
By multiplying $U^{-1}$ from the right, we have
\begin{eqnarray}
&&\left[\begin{array}{cc}
A & -R \\
-Q & -A^{\rm T}
\end{array}\right]
\left[\begin{array}{c}
I_n\\ VU^{-1}
\end{array}\right]
-\left[\begin{array}{c}
\delta_f(U)U^{-1}\\ \delta_f(V)U^{-1}
\end{array}\right]\nonumber\\
&&=\left[\begin{array}{c}
I_n\\ VU^{-1}
\end{array}\right]
U \Lambda U^{-1}.\label{riegh2}
\end{eqnarray}
Next, by multiplying $[VU^{-1} \ -I_n]$ from the left, we obtain
\begin{eqnarray*}
&&\delta_f(V)U^{-1}-VU^{-1}\delta_f(U)U^{-1}\\
&&+VU^{-1}A+A^{\rm T}VU^{-1}-VU^{-1}RVU^{-1}+Q=0.
\end{eqnarray*}
It can be shown that $\delta_f(VU^{-1})=\delta_f(V)U^{-1}-VU^{-1}\delta_f(U)U^{-1}$. Thus, $X:=VU^{-1}$ is a solution to (\ref{ric}).

Next, from the upper half of (\ref{riegh2}), 
\begin{eqnarray}
A-RX=(U \Lambda +\delta_f(U))U^{-1}. \label{clU}
\end{eqnarray}
From Proposition \ref{lreig:p} 2), 
\begin{eqnarray}
{\rm rspec}_f(\Lambda)&=&{\rm rspec}_f((U \Lambda +\delta_f(U))U^{-1})\nonumber\\
&=&{\rm rspec}_f\left(A-RX\right).\label{re1}
\end{eqnarray}

Let $r$ be the maximum number of linearly independent right eigenvectors $w_1,\dots,w_r\in W$ of $\cH$ associated with right eigenvalues $\beta_i$ $(i=1,\dots,r)$.
Since $W$ is an $n$-dimensional subspace, there exist $w_{r+1},\dots,w_n \in W$ such that ${\rm span}_{\cK}\{w_1,\dots,w_n\}=n$ holds.
Let
\begin{eqnarray*}
    \left[\begin{array}{c}
        \hat U \\ \hat V
    \end{array}\right]
    :=    
    \left[\begin{array}{ccc}
        w_1 & \cdots & w_n
    \end{array}\right].
\end{eqnarray*}
Then, there exists $\hat \Lambda \in \cK^{n\times n}$ such that
\begin{eqnarray*}
    &&\left[\begin{array}{cc}
        A & -R \\
        -Q & -A^{\rm T}
    \end{array}\right]
    \left[\begin{array}{c}
        \hat U\\ \hat V
    \end{array}\right]
    -\left[\begin{array}{c}
        \delta_f(\hat U)\\ \delta_f(\hat V)
    \end{array}\right]
    =\left[\begin{array}{c}
        \hat U\\ \hat V
    \end{array}\right]\hat \Lambda,\\
    &&\hspace{5mm} \hat\Lambda
    :=\left[\begin{array}{cc}
        B_{11}&B_{21}\\
        0&B_{22}
    \end{array}\right],
\end{eqnarray*}
where $B_{11}={\rm diag}\{\beta_1,\dots,\beta_r\}$ and $B_{21}\in \cK^{(n-r)\times r}$, $B_{22}\in \cK^{(n-r)\times (n-r)}$ are suitable matrices. Thus, 
\begin{eqnarray}
{\rm rspec}_f(\cH |_W)={\rm rspec}_f(B_{11})={\rm rspec}_f(\hat \Lambda).\label{eigeq2}
\end{eqnarray}
From (\ref{re1}) and (\ref{eigeq2}), it remains to show ${\rm rspec}_f(\Lambda )={\rm rspec}_f(\hat \Lambda)$.
Since both $[U^{\rm T} \  V^{\rm T}]^{\rm T}$ and $[\hat U^{\rm T} \ \hat V^{\rm T}]^{\rm T}$ consist of bases of $W$, there exists a regular matrix $T\in\cK^{n\times n}$ such that
\begin{eqnarray*}
    \left[\begin{array}{c}
        U \\ V
    \end{array}\right]
    =\left[\begin{array}{c}
        \hat U \\ \hat V
    \end{array}\right]T.
\end{eqnarray*}
By substituting this equality into (\ref{riegh}),  
\begin{eqnarray*}
    &&\cH
    \left[\begin{array}{c}
        \hat U\\ \hat V
    \end{array}\right]
    -\left[\begin{array}{c}
        \delta_f(\hat U)\\ \delta_f(\hat V)
    \end{array}\right] =\left[\begin{array}{c}
        \hat U\\ \hat V
    \end{array}\right] ( T \Lambda +\delta_f(T)) T^{-1},
\end{eqnarray*}
which implies $\hat \Lambda = ( T \Lambda + \delta_f(T)) T^{-1}$.
From Proposition \ref{lreig:p} 2), the set of right eigenvalues of $\hat \Lambda$ and $\Lambda$ are equivalent. 

We prove the second parts. Let $\Lambda:=A-RX$. By premultiplying $X$, we have, from (\ref{ric}),
\begin{eqnarray*}
X\Lambda=XA-XRX=-Q-\delta_f(X)-A^{\rm T}X.
\end{eqnarray*}
The above two equations yield
\begin{eqnarray}
\left[\begin{array}{cc}
A & -R \\
-Q & -A^{\rm T}
\end{array}\right]
\left[\begin{array}{c}
I_n\\ X
\end{array}\right]
-\left[\begin{array}{c}
\delta_f(I_n)\\ \delta_f(X)
\end{array}\right]
=\left[\begin{array}{c}
I_n\\ X
\end{array}\right]
\Lambda.\label{hmm}
\end{eqnarray}
Denote $U:=I_n$ and $V:=X$. Then, $U$ is regular, and $X=VU^{-1}$ holds.  Since $U$ is regular, $w_i\in\cK^{2n}$ $(i=1,\dots,n)$ defined by $[w_1,\dots,w_n]:=[U^{\rm T} \ V^{\rm T}]^{\rm T}$ spans $\cK^n$, and thus $W$ in (\ref{Wdef}) is an $n$-dimensional subspace. From (\ref{hmm}), $W$ is $\cH$ invariant. Finally, it can be shown that (\ref{cHinveig}) holds similarly to the proof of the first parts. 
\end{proof}
% % % % % % % % % % % % % % % % % % % % % % % % % % % % % % % % % % 
\begin{secrem}\label{bss:rem}
Solution $X$ does not depend on the choice of basis of $W$. Every basis of $W$ can be represented with regular matrix $T\in\cK^{n\times n}$ as 
\begin{eqnarray*}
\left[\begin{array}{c}
U \\ V
\end{array}\right]T
=\left[\begin{array}{c}
UT \\ VT
\end{array}\right].
\end{eqnarray*}
 Since $(VT)(UT)^{-1}=VU^{-1}=X$ holds, $X$ does not depend on the choice of basis of $W$. \red
\end{secrem}

As demonstrated in Example \ref{contra:ex}, a symmetric and positive definite solution $X$ to a DRE plays an important role in the contraction analysis. However, it is not guaranteed that $X=UV^{-1}$ has such a property for any $n$-dimensional $\cH$ invariant subspace $W$ in (\ref{Wdef}). In general, $X$ is a complex-valued function because nonlinear eigenvalues and eigenvectors of $\cH$, i.e., $U$ and $V$, can be complex-valued functions as in Example \ref{RLC:ex} below. In the next section, we give a characterization of $W$ defining a symmetric and positive definite solution $X$ to a DRE under an assumption for differential Hamiltonian matrix $\cH$.

Theorem \ref{ham:t} is an extension of the eigenvalue method for the ARE because Theorem \ref{ham:t} demonstrates that solutions to DRE (\ref{ric}) can be obtained by using the right eigenvectors of the corresponding differential Hamiltonian matrix $\cH$.
% % % % % % % % % % % % % % % % % % % % % % % % % % % % % % % % % % % %
\begin{secexa}\label{RLC:ex}
Based on Example \ref{contra:ex}, consider a stabilization problem of an RL-circuit with a nonlinear inductor
\begin{eqnarray*}
    &&\hspace{-5mm}\left[\begin{array}{cc} 1 + x_1^2 & 0 \\ 0 & 1 \end{array}\right]\dot{x}=
    -\left[\begin{array}{cc} 1 & -1 \\ -1 & 1 \end{array}\right]x
    +\left[\begin{array}{c} 0\\ 1 \end{array}\right]u.
\end{eqnarray*}
Then, we have
\begin{eqnarray*}
&&\hspace{-5mm}f =  \left[\begin{array}{c} \frac{-x_1+x_2}{1+x_1^2} \\ x_1 - x_2 \end{array}\right], \ B =  \left[\begin{array}{c} 0\\ 1 \end{array}\right],\ R=BB^{\rm T} =  \left[\begin{array}{cc} 0 &0\\ 0&1 \end{array}\right],\\
&&\hspace{-5mm}A = \frac{\partial f}{\partial x} =  \left[\begin{array}{cc} -\frac{1 + 2 x_1 x_2 - x_1^2}{(1+x_1^2)^2} & \frac{1}{1+x_1^2} \\1 & -1 \end{array}\right].
\end{eqnarray*}
For positive definite $Q:={\rm diag} \{3+4 x_1^2+x_1^4,1\}$ for all $x\in\bR^2$, the differential Hamiltonian matrix is
\begin{eqnarray*}
{\cal H}=\left[\begin{array}{cccc}
 -\frac{1 + 2 x_1 x_2 - x_1^2}{(1+x_1^2)^2} & \frac{1}{1+x_1^2} & 0 & 0\\
 1 & -1 & 0 & -1\\
 -(3+4 x_1^2+x_1^4) & 0 & \frac{1 + 2 x_1 x_2- x_1^2}{(1+x_1^2)^2} & -1\\[1.5mm]
 0 & -1 & - \frac{1}{1+x_1^2} & 1
\end{array}\right].
\end{eqnarray*}
The right eigenvalues and eigenvectors of $\cH$ are
\begin{eqnarray*} 
    &&\hspace{-5mm}\beta^1:=-\frac{2+x_1^2}{1+x_1^2}, \ w^1=\left[\begin{array}{c}
        \frac{1}{1+x_1^2}\\
        -1\\
        1+x_1^2\\
        0
    \end{array}\right], \\
    &&\hspace{-5mm}\beta^2:=-2-2 x_1 x_2-x_1^2-x_1^4-\frac{(x_1-x_2)c(x_1,x_1)}{1+x_1^2},\\
    &&\hspace{-5mm}w^2=\left[\begin{array}{c}
       1 \\
       -1-x_1^2-(x_1+x_2) c(x_1,x_1)\\
       (1+x_1^2)(1+x_1^2-(x_1+x_2)c(x_1,x_1)) \\
       -(x_1+x_2)c(x_1,x_1)
    \end{array}\right],
\end{eqnarray*}
where 
\begin{eqnarray*}
&&\hspace{-5mm}c(x_1,x_2)\\
&&\hspace{-5mm}= \sum_{\{a: a^3 + 6 a - x_1^3 - 3 x_1 - 3 x_2=0\}}\textstyle\frac{(1+x_1^2) (a-x_1)^{\frac{a^2+3}{a^2+2}}}{\int_0^{x_1} (1+b^2)(a-b)^{\frac{a^2+3}{a^2+2}}db+1}.
\end{eqnarray*}
On the basis of Theorem~\ref{ham:t}, we define
\begin{eqnarray*}
    &&\hspace{-5mm}U:=\left[\begin{array}{cc}
        \frac{1}{1+x_1^2} & 1\\
        -1 & -1-x_1^2-(x_1+x_2)c(x_1,x_2)
    \end{array}\right],\\ 
   &&\hspace{-5mm}V:=\left[\begin{array}{cc}
        1+x_1^2 & (1+x_1^2)(1+x_1^2-(x_1+x_2)c(x_1,x_2))\\
        0 & -(x_1+x_2)c(x_1,x_2)
    \end{array}\right].
\end{eqnarray*}
Since two of the solutions to $a^3 + 6 a - x_1^3 - 3 x_1 - 3 x_2=0$ are complex-valued functions, $U$ and $V$ are complex-valued functions. 
Then, a solution to DRE (\ref{conlya}) is
\begin{eqnarray*}
    X := VU^{-1} = \left[\begin{array}{cc}
        2 (1+x_1^2)^2 & 1+x_1^2\\
        1+x_1^2 & 1
    \end{array}\right],
\end{eqnarray*}
and $(\partial X_{ij}/\partial x)B=0$. Moreover, $X$ is positive definite for all $x\in\bR^2$ while $U$ and $V$ are complex-valued functions. According to Example~\ref{contra:ex}, the feedback controller
\begin{eqnarray*}
	u =- \int_0^x B^{\rm T} X dx= -\left(x_1 + x_1^3/3 + x_2\right)
\end{eqnarray*}
makes the closed loop system globally incrementally asymptotically stable. 
Fig. \ref{fig:sim} shows a phase portrait of the closed-loop system.
\begin{figure}[tb]
	\begin{center}
\vspace{2mm}\includegraphics[height=4cm]{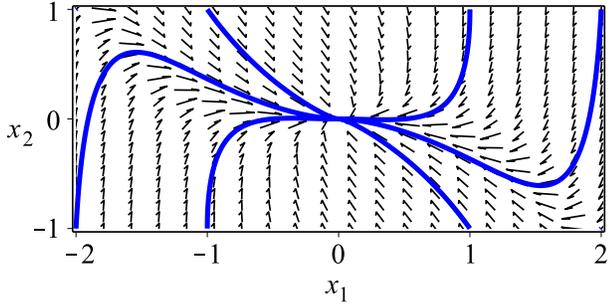}\vspace{-2mm}
		\caption{Phase portrait of the closed-loop system}
		\label{fig:sim}
	\end{center}
\end{figure}
\red
\end{secexa}
% % % % % % % % % % % % % % % % % % % % % % % % % % % % % % % % % % % % % % % %
\section{Detailed Properties in Simple Case}\label{Dp:s}
According to Example \ref{contra:ex} and \cite{MS:13}, a stabilizing solution to the DRE is real symmetric and positive (semi)definite. In the linear case, real symmetricity, regularity, and positive (semi)definiteness depend on a choice of $n$-eigenvectors of the Hamiltonian matrix, i.e., an $n$-dimensional $\cH$-invariant subspace. Here, we study relationships between properties of solutions to the DRE and nonlinear eigenvalues and eigenvectors of the differential Hamiltonian matrix. As a first step, in this paper, we assume that the differential Hamiltonian matrix is simple. 
\begin{secdefn}\label{simp:d}
A matrix $M\in\cK^{n\times n}$ is said to be left (or right) simple if there exist $n$ left (or right) eigenvectors $v_1,\dots,v_n\in\cK^n$ such that ${\rm span}_{\cK}\{v_1,\dots,v_n\}=\cK^n$.
\end{secdefn}

Note that, for any right eigenvector $w\in\cK^{2n}$ of $\cH$, $\{w\} \subset \cK^{2n}$ is a one-dimensional $\cH$-invariant subspace. Therefore, simplicity of $\cH$ implies the existence of the $2n$-dimensional $\cH$-invariant subspace.

It can readily be shown that a matrix $M$ is left (or right) simple if and only if $M$ is $\delta_f$-conjugate to a diagonal matrix, which yields the following proposition.
\begin{secprop}\label{simp:p}
A matrix $M\in\cK^{n\times n}$ is right simple if and only if it is left simple.\red
\end{secprop}

Since left and right simplicity are equivalent properties, a left or right simple matrix is called simple. Also, its left or right eigenvalue is called an eigenvalue.

If the differential Hamiltonian matrix $\cH$ in (\ref{HM}) is simple, it is possible to show the following. 
\begin{secthm}\label{Xprop:thm}
Let $\cH$ be simple. Let $W\subset \cK^{2n}$ be an $n$-dimensional $\cH$ invariant subspace.
\begin{enumerate}
\item\label{digLm} There exist $U,V\in\cK^{n\times n}$ in (\ref{Wdef}) and $\lambda_i$ $(i=1,\dots,n)$ such that $\Lambda:={\rm diag}\{\lambda_1,\dots,\lambda_n\}$ holds in (\ref{riegh}).
\item\label{dcLm} Denote $\lambda^{\delta_f}$ as the set of $\delta_f$-conjugate elements of $\lambda\in\cK$. Also, define for $\lambda_i$ $(i=1,\dots,n)$ in \ref{digLm}),
\begin{eqnarray}
	\{\lambda_1,\dots,\lambda_n\}^{\delta_f}:=\{\lambda\in \lambda_i^{\delta_f}:i=1,\dots,n\}.\label{lamdelt}
\end{eqnarray}
Then, $\{\lambda_1,\dots,\lambda_n\}^\delta_f={\rm rspec}_f (\cH|_W)$.
\item\label{feig} $-\lambda_i, \lambda_i^*, -\lambda_i^* \in {\rm rspec}_f (\cH|_W)$ $(i=1,\dots,n)$, where $\lambda_i^*$ is the complex conjugate of $\lambda_i$.  
\item\label{heluv} If $\cH$ has no eigenvalue on the imaginary axis, then there is at least one $W$ such that $U^*V$ is Hermitian, and $U^{\rm T}V$ is symmetric for $U,V\in\cK^{n\times n}$ in (\ref{Wdef}).
\item\label{regu} Suppose that $U^*V$ is Hermitian or $U^{\rm T}V$ is symmetric for $U,V\in\cK^{n\times n}$ in (\ref{Wdef}). Then, $U$ is regular if and only if there is no $\lambda_i$ $(i=1,\dots,n)$ in \ref{digLm}) satisfying 
\begin{eqnarray}
A^{\rm T}v+\delta_f(v)=-\lambda_i v, \ Rv=0 \label{VREg}
\end{eqnarray}
for non-zero $v\in\cK^n$ such that $[0^{\rm T},v^{\rm T}]^{\rm T}\in W$. Moreover, $V$ is regular if and only if there is no $\lambda_i$ $(i=1,\dots,n)$ in \ref{digLm}) satisfying  
\begin{eqnarray}
A u-\delta_f(u)=\lambda_i u, \ Q u=0 \label{UVR}
\end{eqnarray}
for non-zero $u\in\cK^n$ such that $[u^{\rm T} \ 0^{\rm T}]^{\rm T}\in W$.
\item\label{stab} Suppose that $U,V$, and $\Lambda$ are chosen as in \ref{digLm}). Denote the real and imaginary parts of $\lambda_i(x,t)$ by ${\rm Re}(\lambda_i)$ and ${\rm Im}(\lambda_i)$, respectively. Suppose that $U,U^{-1}$, and $V$ are defined in $\bR^n\times \bR$. Suppose that there is symmetric and positive semidefinite $\bar Q\in\bR^{n\times n}$ such that $Q\ge \bar Q$ for all $(x,t)\in\bR^n\times\bR$. If for some $c<0$, ${\rm Re}(\lambda_i)\le c$ $(i=1,\dots,n)$ for all $(x,t)\in \bR^n\times \bR$ then $X:=VU^{-1}$ is symmetric and positive semidefinite for all $(x,t)\in \bR^n\times \bR$.\red
\end{enumerate}
\end{secthm}
%-------------------------------------------

Theorem \ref{Xprop:thm} \ref{heluv}) and \ref{regu}) give characterizations of symmetricity and regularity of a solution to the DRE. Denote $\Omega:=U^*V$ and $\tilde \Omega:=U^{\rm T}V$. If $U$ is regular, we have
\begin{eqnarray}
X=VU^{-1}=(U^{-1})^*\Omega U^{-1}=(U^{-1})^{\rm T}\tilde{\Omega} U^{-1}. \label{XVUtr}
\end{eqnarray}
Thus, $X$ is real symmetric if both $\Omega=\Omega^*$ and $\tilde \Omega = \tilde \Omega^{\rm T}$ hold and $U$ is regular. Regularity of $U$ is characterized by (\ref{VREg}).

Conditions (\ref{VREg}) and (\ref{UVR}) can be viewed as generalizations of Popov-Belevitch-Hautus (PBH) accessibility and observability tests on nonlinear systems, respectively.  In fact, there is no $\lambda$ such that (\ref{VREg}) and (\ref{UVR}) hold if $\dot x = f(x) + B u$ is locally strongly accessible \cite{NJ:90} and if $\dot x = f(x)$, $y=h(x)$ is locally observable \cite{NJ:90}, when $A=\partial f/\partial x$, $R=BB^{\rm T}$, and $Q=\partial h/\partial x$ \cite{PBHKO}. Thus, if these two systems are accessible and observable as in Example \ref{RLC:ex}, and if $\cH$ has no eigenvalue on the imaginary axis, then the DRE has at least one real symmetric and regular solution. Moreover, if the condition in Theorem~\ref{Xprop:thm}~\ref{stab}) holds, one of the real symmetric solutions is positive definite. 

The remainder is dedicated to the proof of Theorem~\ref{Xprop:thm}.
% % % % % % % % % % % % % % % % % % % % % % % % % %
\subsection{Proofs of \ref{digLm}) and \ref{dcLm})}
Although the number of linearly independent right eigenvectors of differential Hamiltonian matrix $\cH$ is at most $2n$, the number of eigenvalues can be infinite, which is different from the eigenvalues in linear algebra. Consider the right eigenvalue $\lambda \in \cK$ and its associated right eigenvector $w \in  \cK^{2n}\setminus\{0\}$ of $\cH$.
For $a\in{\mathcal K}\setminus\{0\}$, from the definition of the right eigenvalue and eigenvector, we have
\begin{eqnarray}
\cH aw-\delta_f(aw)&=&\cH aw-a\delta_f(w)-\delta_f(a)w\nonumber\\
&=& (\lambda-\delta_f(a)a^{-1}) a w\label{dcd}
\end{eqnarray}
Thus, $\lambda-\delta_f(a)a^{-1}$ and $a w$ are also right eigenvalue and eigenvector, respectively. These $\lambda$ and $\lambda-\delta_f(a)a^{-1}$ are $\delta_f$-conjugate. 
%----------------------------------------------

Consider differential Hamiltonian matrix $\cH$. An $n$-dimensional $\cH$ invariant subspace $W\subset \cK^{2n}$ can always be generated by linearly independent $n$ right eigenvectors, which is demonstrated here. Let $W$ be generated by $w_1,\dots,w_n$, and let the column elements of $\hat W_2:=[\hat w_{n+1},\dots,\hat w_{2n}]$ be $n$ right eigenvectors associated with eigenvalues $\lambda_i$ $(i=n+1,\dots, 2n)$ such that 
\begin{eqnarray*}
	{\rm span}_{\cK}\{w_1,\dots,w_n,\hat w_{n+1},\dots,\hat w_{2n}\}=\cK^{2n}.
\end{eqnarray*}
Such $\hat W_2$ always exists because of the simplicity of $\cH$. From the definitions of the $n$-dimensional $\cH$ invariant subspace and the right eigenvalue and eigenvector, we have
\begin{eqnarray}
&&\cH \left[\begin{array}{cc} W&\hat{W}_2\end{array}\right]-\left[\begin{array}{cc} \delta_f(W) & \delta_f(\hat{W}_2)\end{array}\right]\nonumber\\
&&=\left[\begin{array}{cc} W & \hat{W}_2\end{array}\right]\left[\begin{array}{cc} A_{11} & 0 \\ 0 & A_{22}\end{array}\right],\label{ch2n}
\end{eqnarray}
where $A_{11}\in\cK^{n\times n}$ is a suitable matrix, and $A_{22}={\rm diag}\{ \lambda_{n+1},\dots,\lambda_{2n}\}$; consequently
\begin{eqnarray*}
	&&\left[\begin{array}{cc} W & \hat{W}_2 \end{array}\right]^{-1}\left(\cH \left[\begin{array}{cc} W&\hat{W}_2\end{array}\right]-\left[\begin{array}{cc} \delta_f(W) &\delta_f(\hat{W}_2)\end{array}\right]\right)\\
	&&=\left[\begin{array}{cc} A_{11} & 0 \\ 0 & A_{22}\end{array}\right].
\end{eqnarray*}
Since $\cH$ is simple, $A_{11}$ is also simple. Let column elements of $\hat W_1:=[\hat w_1,\dots,\hat w_n]$ be linearly independent $n$ right eigenvectors of $A_{11}$ associated with eigenvalues $\lambda_i$ $(i=1,\dots, n)$. Also, denote $\Lambda :={\rm diag}\{ \lambda_1,\dots,\lambda_n\}$. From the definition of right eigenvalues and eigenvectors, we obtain $A_{11} \hat W_1- \delta_f (\hat W_1)=\hat W_1 \Lambda$. From this equality and (\ref{ch2n}),
\begin{eqnarray*}
	\cH W \hat W_1 - \delta_f(W) \hat W_1 = W A_{11} \hat W_1 = W (\hat W_1 \Lambda + \delta_f (\hat W_1)),
\end{eqnarray*}
and thus $\cH W \hat W_1- \delta_f(W\hat W_1) = W \hat W_1 \Lambda$. Because of $\Lambda ={\rm diag}\{ \lambda_1,\dots,\lambda_n\}$, all column elements of regular matrix $W \hat W_1$ are right eigenvectors of $\cH$. In summary, an $n$-dimensional $\cH$ invariant subspace can always be generated by linearly independent $n$ right eigenvectors if $\cH$ is simple, which implies that $\{\lambda_1,\dots,\lambda_n\}^\delta_f={\rm rspec}_f (\cH|_W)$ holds. Therefore, for simple $\cH$, the set ${\rm rspec}_f (\cH|_W)$ is obtained by finding $n$ linearly independent right eigenvectors in $W$ while the number of elements in ${\rm rspec}_f (\cH|_W)$ can be infinite. Note that a solution to the DRE is uniquely determined irrespective of the choice of eigenvalues in ${\rm rspec}_f (\cH|_W)$. Let $\hat \lambda_i$ be $\delta_f$-conjugate to $\lambda_i$ $(i=1,\dots,n)$. Then, there exists $a_i$ such that $\hat \lambda_i = \lambda_i- \delta_f (a_i)/a_i$. Define $\hat \Lambda:={\rm diag}\{ \hat \lambda_1,\dots,\hat \lambda_n\}$ and $\hat A={\rm diag}\{ a_1,\dots, a_n\}$.
In a similar manner to the discussion in (\ref{dcd}), $\cH W \hat W_1 \hat A - \delta_f(W\hat W_1 \hat A)  = W \hat W_1 \hat A \hat \Lambda$. Owing to Remark~\ref{bss:rem}, $W \hat W_1$ and $W \hat W_1 \hat A$ give the same solution $X$.  
% % % % % % % % % % % % % % % % % % % % % % % % %
\subsection{Proof of \ref{feig})}
Owing to the specific structure of $\cH$, we have the following relationship between the left and right eigenvalues of $\cH$, where $\cH$ does not need to be simple. . 
%---------------------------------------------
\begin{secprop}\label{eigch:p}
$\beta\in\cK$ is a right eigenvalue of $\cH$ if and only if $-\beta$ is its left eigenvalue, or equivalently, if and only if $-\beta^*$ is its left eigenvalue, or equivalently, if and only if $\beta^*$ is its right eigenvalue.
\end{secprop}
%---------------------------------------------
\begin{proof}
First, we show that if $\beta\in\cK$ is a right eigenvalue, $-\beta$ is a left eigenvalue. Let $w\in\cK^{2n}$ be a right eigenvector associated with right eigenvalue $\beta$, i.e., $w$ and $\beta$ satisfy
\begin{eqnarray}
\cH w- \delta_f(w)=\beta w.\label{chreig}
\end{eqnarray}
For matrix $J\in\cK^{2n\times 2n}$,
\begin{eqnarray}
J:=\left[\begin{array}{cc}
0&I_n\\-I_n&0
\end{array}\right]\label{J},
\end{eqnarray}
we have $J^{-1}\cH =-\cH^{\rm T}J^{-1}$. By premultiplying $J^{-1}$ with both sides of (\ref{chreig}), we have
\begin{eqnarray*}
-\cH^{\rm T} J^{-1}w-\delta_f(J^{-1}w)=\beta J^{-1}w.
\end{eqnarray*}
Therefore, $-\beta$ is a left eigenvalue of $\cH$ with left eigenvector $J^{-1}w$, and the converse can readily be shown.

Since $\cH$ is real analytic, by taking the conjugate transpose instead of the transpose in the above equations, we can show that $\beta\in\cK$ is a right eigenvalue if and only if $-\beta^*$ is a left eigenvalue.
Finally, from the above proof, $\hat \beta:=-\beta^*$ is a left eigenvalue if and only if $-\hat \beta:=\beta^*$ is a right eigenvalue.
\end{proof}

Now, we are ready to prove \ref{feig}).
\begin{proof}[Proof of \ref{feig})]
Let $\cH\in\cK^{2n\times 2n}$ be simple, and let $\lambda\in\cK$ be its eigenvalue. From Proposition~\ref{eigch:p}, $-\lambda$, $-\lambda^*$, and $\lambda^*$ are also eigenvalues.
\end{proof}
% % % % % % % % % % % % % % % % % % % % % % % % % % % % % %
\subsection{Proof of \ref{heluv})}
Let $\omega_{i,j}$ and $\tilde \omega_{i,j}$ be the $(i,j)$ elements of $\Omega:=U^*V$ and $\tilde \Omega:=U^{\rm T}V$, respectively, i.e., 
\begin{eqnarray}
\omega_{i,j}:=u_i^*v_j,\ \tilde \omega_{i,j}:=u_i^{\rm T}v_j, \ i,j=1,2,\dots,n,\label{omgijdef}
\end{eqnarray}
Conditions $\Omega=\Omega^*$ and $\tilde \Omega=\tilde \Omega^{\rm T}$ can be rewritten as
\begin{eqnarray}
&&\hspace{-10mm}\omega_{i,j}-\omega_{j,i}^*=u_i^*v_j-v_i^*u_j=0, \ i,j=1,2,\dots,n,\label{elome}\\
&&\hspace{-10mm}\tilde{\omega}_{i,j}-\tilde{\omega}_{j,i}^{\rm T}=u_i^{\rm T}v_j-v_i^{\rm T}u_j=0, \ i,j=1,2,\dots,n.\label{tpome}
\end{eqnarray}
These conditions are characterized by eigenvalues of $\cH$.
%---------------------------------------------
\begin{secprop}\label{rsym:p}
Let $w_i=[u_i^{\rm T} \ v_i^{\rm T}]^{\rm T}$ and $w_j=[u_j^{\rm T} \ v_j^{\rm T}]^{\rm T}\in\cK^{2n}$ be right eigenvectors associated with right eigenvalues $\lambda_i$ and $\lambda_j\in\cK$ of $\cH$, respectively.
If $\lambda_i^*$ and $-\lambda_j$ $(i,j=1,\dots,n)$ are not $\delta_f$-conjugate, (\ref{elome}) holds.
Also, if $\lambda_i$ and $-\lambda_j$ $(i,j=1,\dots,n)$ are not $\delta_f$-conjugate, (\ref{tpome}) holds.
\end{secprop}
%---------------------------------------------
\begin{proof}
We prove that the non $\delta_f$-conjugacy of $\lambda_i^*$ and $-\lambda_j$ $(i,j=1,\dots,n)$ implies (\ref{elome}) by contraposition.
That is, we show that $\omega_{i,j}-\omega_{j,i}^*\neq 0$ implies that $\lambda_i^*$ and $-\lambda_j$ are $\delta_f$-conjugate.
For $J$ in (\ref{J}), $\cH^{\rm T}J+J\cH=0$ holds.
Since elements of $\cH$ are real analytic functions, the definition of the right eigenvalue and eigenvector 
\begin{eqnarray}
\cH w_j=\delta_f(w_j)+\lambda_j w_j\label{chw}
\end{eqnarray}
yields
\begin{eqnarray}
w_i^* \cH^{\rm T}=\delta_f(w_i^*)+\lambda_i^* w_i^*.\label{chwel}
\end{eqnarray}
By computing $w_i^*(\cH^{\rm T}J+J\cH)w_j$ with (\ref{chw}) and (\ref{chwel}), we have
\begin{eqnarray*}
&&w_i^*(\cH^{\rm T}J+J\cH)w_j\\
&&=(\delta_f(w_i^*)+\lambda_i^* w_i^*)J w_j+w_i^*J(\delta_f(w_j)+\lambda_j w_j)=0.
\end{eqnarray*}
From (\ref{J}) and $w_i=[u_i^{\rm T} \ v_i^{\rm T}]^{\rm T}$, we have
\begin{eqnarray*}
&&(\delta_f(w_i^*)+\lambda_i^* w_i^*)J w_j+w_i^*J(\delta_f(w_j)+\lambda_j w_j)\\	
&&=\delta_f(u_i^*v_j-v_i^*u_j)+(\lambda_i^*+\lambda_j) (u_i^*v_j - v_i^*u_j).
\end{eqnarray*}
From (\ref{omgijdef}) and $u_i^*v_j-v_i^*u_j=\omega_{i,j}-\omega_{j,i}^*\neq 0$, the equality can be rewritten as
\begin{eqnarray*}
&&\delta_f(\omega_{i,j}-\omega_{j,i}^*)+(\lambda_i^*+\lambda_j) (\omega_{i,j}-\omega_{j,i}^*)=0,\\
&&\lambda_i^*+\delta_f(\omega_{i,j}-\omega_{j,i}^*)(\omega_{i,j}-\omega_{j,i}^*)^{-1}=-\lambda_j.
\end{eqnarray*}
Thus, $\lambda_i^*$ and $-\lambda_j$ are $\delta_f$-conjugate.
In a similar manner, it is possible to show that (\ref{tpome}) holds if $\lambda_i$ and $-\lambda_j$ $(i,j=1,\dots,n)$ are not $\delta_f$-conjugate.
\end{proof}

To analyze further, we investigate the $\delta_f$-conjugacy of pairs of $(\lambda, -\lambda)$ and $(\lambda, -\lambda^*)$.
\begin{secprop}\label{Himg1:prop}
Differential Hamiltonian matrix $\cH$ has no left (or right) nonlinear eigenvalue on the imaginary axis if and only if for any left (or right) nonlinear eigenvalue $\lambda$ of $\cH$, neither pair $(\lambda, -\lambda)$ nor $(\lambda, -\lambda^*)$ is $\delta_f$-conjugate. 
\end{secprop}
%---------------------------------------------
\begin{proof}
(Necessity)
We prove by contraposition. First, suppose that $(\lambda, -\lambda)$ is $\delta$-conjugate. Then, there exists non-zero $a\in\cK$ such that $2\lambda=\delta(a)/a$, which implies $\lambda=\delta(a^{1/2})/a^{1/2}$. Thus, ($\lambda$,$0$) is $\delta_f$-conjugate with respect to $a^{1/2}$. From Proposition \ref{lreig:p} 1), $0$ is a left (or right) eigenvalue of $\cH$. Next, suppose that $(\lambda, -\lambda^*)$ is $\delta_f$-conjugate. Then, $2{\rm Re}(\lambda)=\lambda+\lambda^*=\delta(a)/a$ for some non-zero $a\in\cK$. Compute $4{\rm Re}(\lambda)=2{\rm Re}(\lambda)+2{\rm Re}(\lambda)^*=\delta(a)/a+\delta(a^*)/a^*=\delta(aa^*)/(aa^*)$, where $aa^*$ is real valued, and consequently $4\lambda - \delta(aa^*)/(aa^*)= 4j{\rm Im}(\lambda)$. Thus, ($\lambda$,$j{\rm Im}(\lambda)$) is $\delta_f$-conjugate with respect to $(aa^*)^{1/4}$. Therefore, $\cH$ has a left (or right) eigenvalue on the imaginary axis. 

(Sufficiency)
We prove by contraposition. Let $\lambda$ be a left (or right) eigenvalue of $\cH$ on the imaginary axis. Then, $\lambda=-\lambda^*$. That is, $(\lambda,-\lambda^*)$ is $\delta_f$-conjugate. Moreover, if $\lambda=0$, $(\lambda,-\lambda)$ is $\delta_f$-conjugate.
\end{proof}
%---------------------------------------------

Now, we are ready to prove \ref{heluv}).
\begin{proof}[Proof of \ref{heluv})]
Let $\{w_1,\dots,w_{2n}\}$ be the set of linearly independent eigenvectors of $\cH$ associated with eigenvalues $\lambda_1,\dots,\lambda_{2n}$. Here, we show that $w_1,\dots,w_n$ can be chosen such that neither $(\lambda_i,-\lambda_j^*)$ nor $(\lambda_i,-\lambda_j)$ is $\delta_f$-conjugate for any $i,j=1,\dots,n$. Then, Proposition \ref{rsym:p} implies that $U^*V$ is Hermitian, and $U^{\rm T}V$ is symmetric for $W={\rm span}_{\cK}\{w_1,\dots,w_n\}$.

Let $\{a_1,\dots,a_r\}$ be the set of eigenvalues, where $(a_i,a_j)$ is not $\delta_f$-conjugate for any $i \neq j$, such that each $\lambda_i$ $(i=1,\dots,2n)$ is $\delta_f$-conjugate to one of its elements. First, we focus on $a_1$. According to Theorem \ref{Xprop:thm} \ref{feig}), $-a_1$, $a_1^*$, and $-a_1^*$ are also eigenvalues of $\cH$. From Proposition \ref{Himg1:prop}, $(a_1,-a_1)$ is not $\delta_f$-conjugate. That is, one of $a_2,\dots,a_r$ can be chosen as $-a_1$. Here, we chose $a_2=-a_1$ without loss of generality. Moreover, if $(a_1,a_1^*)$ is not $\delta_f$-conjugate, none of pair $(b,c)$ $(b\neq c; b,c\in\{a_1,-a_1,a_1^*,-a_1^*\})$ is $\delta_f$-conjugate. Then, we can choose $a_3=a_1^*$ and $a_4=-a_1^*$ without loss of generality. 

We perform a similar procedure for $a_5,\dots,a_r$. Then, we notice that $r$ is an even number, i.e. $r=2\hat r$ for some $\hat r$. Consider $\{a_1,a_3,\dots,a_{2\hat r-1}\}$. Then, neither $(a_i,-a_j)$ nor $(a_i,-a_j^*)$ $(i,j=1,3,\dots,2\hat r-1)$ is $\delta_f$-conjugate. Also, for the set $\{a_2,a_4,\dots,a_{2\hat r}\}$, neither $(a_i,-a_j)$ nor $(a_i,-a_j^*)$ $(i,j=2,4,\dots,2\hat r)$ is $\delta_f$-conjugate. Therefore, if we construct $W$ by using the eigenvectors of $\cH$ associated with the eigenvalues, which are $\delta_f$-conjugate to one of $a_1,a_3,\dots,a_{2\hat r-1}$, or the eigenvectors associated with the eigenvalues, which are $\delta_f$-conjugate to one of $a_2,a_4,\dots,a_{2\hat r}$, then $\lambda_i$ $(i=1,\dots,n)$ satisfy the conditions in Proposition \ref{rsym:p}.
\end{proof}
% % % % % % % % % % % % % % % % % % % % % % % % % % % % % %
\subsection{Proof of \ref{regu})}
\begin{proof}[Proof of \ref{regu})]
Here, we prove \ref{regu}) only for regularity of $V$ when $U^* V$ is Hermitian. In a similar manner, we can prove the other cases. 

(Sufficiency)
We prove this by contraposition. Let $V$ be not regular. There exists a non-zero $v$ such that
\begin{eqnarray}
V v=0.\label{uvze}
\end{eqnarray}
The lower half of (\ref{riegh}) is $-QU -A^{\rm T}V- \delta_f(V)=V\Lambda$.
By multiplying $v$, we have, from (\ref{uvze}),
\begin{eqnarray}
-QU v - \delta_f(V) v=V \Lambda v.\label{HxUv}
\end{eqnarray}
Note that from (\ref{uvze}), $\delta_f(Vv)=\delta_f(V)v+V\delta_f(v)=0$ holds, which yields $-\delta_f(V)v=V\delta_f(v)$.
By using this, (\ref{HxUv}) can be rewritten as 
\begin{eqnarray}
-QU v +V\delta_f(v)=V\Lambda v.\label{eigeqch}
\end{eqnarray}
By premultiplying $v^*U^*$,  from  $U^*V=V^*U$, we obtain
\begin{eqnarray*}
-v^*U^*QU v +v^*V^*U\delta_f(v)=v^*V^*U\Lambda v.
\end{eqnarray*}
Since $Vv=0$, the above equation implies
\begin{eqnarray}
QU v=0,\label{huv0}
\end{eqnarray}
and thus, from (\ref{eigeqch}),
\begin{eqnarray}
V(\Lambda v-\delta_f(v))=0.\label{vlvd0}
\end{eqnarray}
Note that (\ref{huv0}) and (\ref{vlvd0}) hold for all $v$ satisfying $Vv=0$.

Next, we show the existence of $\lambda$ and non-zero $\hat{v}$ satisfying $V\hat{v}=0$ and
\begin{eqnarray}
\Lambda \hat{v}-\delta_f(\hat{v})=\lambda \hat{v}. \label{lvdvlv}
\end{eqnarray}
We assume that $v_1$, the first element of $v$, is non-zero. Then, from (\ref{uvze}) we have $V(v/v_1)=0$ and from (\ref{vlvd0}),
\begin{eqnarray*}
(1/v_1)V(\Lambda v-\delta_f(v))=V(\Lambda (v/v_1) -  \delta_f(v)/v_1)=0.
\end{eqnarray*}
Also, by using $\delta_f(v/v_1)=\delta_f(v)/v_1+\delta_f(1/v_1)v$ and (\ref{uvze}), we obtain
\begin{eqnarray*}
&&V(\Lambda (v/v_1) - \delta_f(v/v_1)+\delta_f(1/v_1)v)\\
&&=V(\Lambda (v/v_1) - \delta_f(v/v_1))\\
&&=V(\Lambda (v/v_1) - \delta_f(v/v_1)-\lambda_1 (v/v_1))=0.
\end{eqnarray*}
This equality can also be expressed as $V\bar v=0$, where
\begin{eqnarray*}
\bar v:=
\left[\begin{array}{c}\lambda_1\\ \lambda_2 (v_2/v_1)\\ \vdots \\ \lambda_n (v_n/v_1) \end{array}\right]
-\left[\begin{array}{c}0\\ \delta_f(v_2/v_1)\\ \vdots \\ \delta_f(v_n/v_1)
\end{array}\right]-\left[\begin{array}{c}\lambda_1\\ \lambda_1 (v_2/v_1)\\ \vdots \\ \lambda_1 (v_n/v_1)\end{array}\right].
\end{eqnarray*}
If $\bar v=0$, let $\hat v:=v/v_1$ and $\lambda:=\lambda_1$.
Then, $\hat v$ and $\lambda$ satisfy $V\hat{v}=0$ and (\ref{lvdvlv}).
Otherwise, let $v:=\bar v$.
Then $v_1$, the first element of $v$, is zero.
This $v$ satisfies $Vv=0$ and thus (\ref{vlvd0}).
We assume that $v_2$, the second element of $v$, is non-zero and repeat the above procedure for $v$.
Finally,  there exists $i\le n$ such that $v=[0 \ \cdots \ 0 \ v_i \ 0 \cdots \  0]^{\rm T}$ $(v_i\neq 0)$.
For $\hat{v}:=v/v_i$ and $\lambda=\lambda_i$, $V\hat{v}=0$ and (\ref{lvdvlv}) hold.
In summary, there exist $\lambda$ and non-zero $\hat{v}$ satisfying $V\hat{v}=0$ and (\ref{lvdvlv}).

From the upper half of (\ref{riegh}), $V\hat{v}=0$ and (\ref{lvdvlv}), we have
\begin{eqnarray*}
&&A U\hat{v}+ RV\hat{v}-\delta_f(U)\hat{v}=U\Lambda \hat{v},\\
&&A U\hat{v}-\delta_f(U)\hat{v}-U\delta_f(\hat{v})=\lambda U\hat{v},\\
&&A U\hat{v}-\delta_f(U\hat{v})=\lambda U\hat{v},
\end{eqnarray*}
where $U\hat{v}\neq 0$. Otherwise, $[U^{\rm T} \ V^{\rm T}]^{\rm T}\hat v=0$, i.e., the column vectors of $[U^{\rm T} \ V^{\rm T}]^{\rm T}$ are linearly dependent, which contradicts that $W\subset\cK^{2n}$ in Theorem \ref{ham:t} is an $n$-dimensional subspace.
Since $\hat{v}$ satisfies (\ref{huv0}), i.e., $QU\hat v=0$, (\ref{UVR}) holds for $w:=U\hat{v}$ and $\lambda$.

(Necessity)
Here, we prove by contraposition. That is, we show that if there is some $\lambda_i$ in Theorem \ref{Xprop:thm} \ref{digLm}) satisfying (\ref{UVR}) for non-zero $u\in\cK^n$ such that $[u^{\rm T} \ 0^{\rm T}]^{\rm T}\in W$, then $V$ is not regular. Let $w_i\in W$ $(i=1,\dots,n)$ be a right eigenvector of $\cH$ associated with an eigenvalue $\lambda_i$ $(i=1,\dots,n)$. If we choose $\lambda_i$ $(i=1,\dots,n)$ such that Theorem \ref{Xprop:thm} \ref{digLm}) holds, we have
\begin{eqnarray}
	\left[\begin{array}{c}
		U \\ V
	\end{array}\right]
	=\left[\begin{array}{ccc}
	w_1&\cdots&w_n
	\end{array}\right].\label{UVw}
\end{eqnarray}
for $U,V$ in (\ref{Wdef}).  In fact, one of $w_i$ can be chosen as $[u^{\rm T} \ 0^{\rm T}]^{\rm T}$ as follows. For $\lambda_i$ and non-zero $u$ satisfying (\ref{UVR}), we have
\begin{eqnarray*}
\left[\begin{array}{cc}
A & R \\
-Q & -A^{\rm T}
\end{array}\right]
\left[\begin{array}{c}
u\\ 0
\end{array}\right]
-\left[\begin{array}{c}
\delta_f(u)\\ 0
\end{array}\right]
=\lambda_i
\left[\begin{array}{c}
u\\ 0
\end{array}\right],
\end{eqnarray*}
which implies that $[u^{\rm T} \ 0^{\rm T}]^{\rm T}$ is a right eigenvector of $\cH$ associated with $\lambda_i$. Furthermore, since $[u^{\rm T} \ 0^{\rm T}]^{\rm T}\in W$, one of $w_i$ can be chosen as $w_i=[u^{\rm T} \ 0^{\rm T}]^{\rm T}$. Then, $V$ is not regular for a basis $\{w_1,\dots,w_n\}$ of $W$. Note that from Remark~\ref{bss:rem}, regularity of $V$ does not depend on the choice of basis.
\end{proof}
% % % % % % % % % % % % % % % % % % % % % % % % % % % % % % % % % % % % % % % % % % % % % % % % %
\subsection{Proof of \ref{stab})}
\begin{proof}[Proof of \ref{stab})]
From (\ref{XVUtr}), if $U$ is regular, positive semidefiniteness of $X$ and $U^*V$ are equivalent. Here, we prove positive semidefiniteness of $U^*V$. By respectively multiplying the upper and lower parts of (\ref{riegh}) by $V^*$ and $U^*$ from the left,
\begin{eqnarray}
&& V^* A U - V^* R V - V^* \delta_f (U) = V^* U \Lambda, \label{V}\\
&& - U^* Q U - U^* A^{\rm T} V -U^*  \delta_f (V) = U^* V \Lambda. \label{U}
\end{eqnarray}
By adding the complex conjugate of (\ref{U}) to (\ref{V}),
\begin{eqnarray}
\delta_f (V^* U) +V^* U \Lambda + \Lambda^* V^* U = - V^* R V - U^* Q U.\label{lyacon}
\end{eqnarray}
From the assumption for $Q$, there exists a symmetric and positive semidefinite matrix $\bar U\in\bR^{n\times n}$ such that $-V^* R V + U^* Q U \le -\bar U$ for all $(x,t)\in\bR^n\times\bR$. Consider linear time-varying system $d \delta  z/dt = \Lambda (\phi(x_0,t),t) \delta z$ along trajectory $\phi(x_0,t)$ of $\dot x = f(x,t)$ with the initial condition $x(t_0)=x_0$. Then, we have, from (\ref{lyacon}),
\begin{eqnarray}
&&\frac{d}{dt}(\delta z^*(t)V^*(\phi(x_0,t),t)U(\phi(x_0,t),t) \delta z(t))\nonumber\\
&& \le - \delta z^*(t) \bar U \delta z(t).\label{diffinW}
\end{eqnarray}
Since ${\rm Re}(\lambda_i)<c$ $(i=1,\dots,n)$, i.e., ${\rm Re}(\Lambda)< c I_n$ for all $(x,t)\in \bR^n\times \bR$, the linear time-varying system is uniformly asymptotically stable at the origin \cite{Khalil:96}. Therefore the time integral of (\ref{diffinW}) is
\begin{eqnarray*}
&&\delta z^*(t_0)V^*(x_0,t_0)U(x_0,t_0) \delta z(t_0)\\
&&= \int_{t_0}^{\infty} \delta z^*(t) \bar U \delta z(t) dt \ge 0
\end{eqnarray*}
for any $\delta z(t_0)\in\bR^n$.
Therefore, $V^*U$ is symmetric and positive semidefinite at each $(x_0,t_0)\in \bR^n\times \bR$.
\end{proof}
% % % % % % % % % % % % % % % % % % % % % % % % % % % % % % % % % % % % % % % % % % % % % % % % %
\section{Conclusion}\label{Con:s}
In this paper, we presented a nonlinear eigenvalue method for the DRE for contraction analysis. First, we showed that all solutions to the DRE can be expressed as functions of nonlinear eigenvectors of the corresponding differential Hamiltonian matrix. Next, in the simple case, we studied solution structures, e.g. real symmetricity and regularity. Future work includes relaxing the simplicity assumption and constructing methods for finding nonlinear eigenvectors of the Hamiltonian matrix. As a solution method to the HJE, the generating function method \cite{park2006determination,park2006solving} is known. For the ARE, this method is useful for finding other eigenvectors of a Hamiltonian matrix from its eigenvectors, and this method may be extended to the DRE.
% % % % % % % % % % % % % % % % % % % % % % % % % % % % % % % % % % % % % % % % % % % % % % % % %
\bibliographystyle{IEEEtran} 
\bibliography{hceref}
\end{document}